\documentclass[english]{amsart}%
\usepackage{amsfonts}
\usepackage{amsmath,color}
\usepackage{amssymb}
\usepackage{amsthm}
\usepackage{amscd}
\usepackage{babel}
\usepackage[latin1]{inputenc}
\usepackage{graphicx}
\usepackage{amsmath}%
\providecommand{\U}[1]{\protect\rule{.1in}{.1in}}
\newtheorem{teor}{Theorem}

\newtheorem{prop}{Proposition}
\newtheorem{con}{Conjecture}
\newtheorem{lem}{Lemma}

\theoremstyle{definition}

\renewcommand{\subjclassname}{AMS \textup{2010} Mathematics Subject
Classification\ }

\begin{document}

\author{J.M. Grau Ribas}
\address{Departamento de Matemáticas, Universidad de Oviedo\\
Avda. Calvo Sotelo s/n, 33007 Oviedo, Spain}
\email{grau@uniovi.es}

\title{Concerning an adversarial version of the Last-Success-Problem}

\begin{abstract}
There are $n$ independent Bernoulli random variables with parameters $p_i$ that are observed sequentially. Two players, A and B, act in turns starting with player A. Each player has the possibility on his turn, when $I_k=1$, to choose whether to continue with his turn or to pass his turn on to his opponent for observation of the variable $I_{k+1}$. If $I_k=0$, the player must necessarily to continue with his turn. After observing the last variable, the player whose turn it is wins if $I_n=1$, and loses otherwise. We determine the optimal strategy for the player whose turn it is and establish the necessary and sufficient condition for player A to have a greater probability of winning than player B. We find that, in the case of $n$ Bernoulli random variables with parameters $1/n$, the probability of player A winning is decreasing with $n$ towards its limit $\frac{1}{2} - \frac{1}{2\,e^2}=0.4323323...$. We also study the game when the parameters are the results of uniform random variables, $\mathbf{U}[0,1]$.
\end{abstract}

\maketitle

\subjclassname{60G40, 62L15}

\section{Introduction}

The Last-Success-Problem (LSP) is the problem of maximizing the probability
of stopping on the last success in a finite sequence of Bernoulli trials.
There are $n$ Bernoulli random variables which are observed sequentially.
The problem is to find a stopping rule to maximize the probability of stopping at the last "1". This problem has been studied by Hill and Krengel
\cite{1992} and Hsiau and Yang \cite{2000} for the case in which the random
variables are independent and was simply and elegantly solved by T.F. Bruss
in \cite{BR1} with the following famous result.

\begin{teor}
(Odds-Theorem, T.F. Bruss 2000). Let $I_{1},I_{2},...,I_{n}$ be $n$
independent Bernoulli random variables with known $n$. We denote by ($%
i=1,...,n$) $p_{i}$ the parameter of $I_{i}$; i.e. ($p_{i}=P(I_{i}=1)$). Let
$q_{i}=1-p_{i}$ and $r_{i}=p_{i}/q_{i}$. We define the index

\begin{equation*}
\mathbf{s}=
\begin{cases}
\max\{1\leq k\leq n: \sum_{j=k}^n r_j \geq 1\}, & \text{if $\sum_{i=1}^n
r_i\geq 1$ }; \\
1, & \text{ otherwise }%
\end{cases}%
\end{equation*}

To maximize the probability of stopping on the last $"1"$ in the sequence,
it is optimal to stop on the first $"1"$ that we encounter among the
variables $I_{\mathbf{s}},I_{\mathbf{s}+1},...,I_{n}$.

The optimal win probability is given by

$$\mathcal{V}(p_1,...,p_n):=\displaystyle{\ \left( \prod_{j=\mathbf{s} }^{n}q_j \right) } \displaystyle{%
\left(\sum_{i=\mathbf{s} }^{n} r_i \right)}$$
\end{teor}

We propose the following adversarial version of the problem in this paper.
There are $n$ independent Bernoulli random variables $I_{i}$ with parameters
$p_{i}$ that are observed sequentially. Two players, A and B, act in turns
starting with A. After observing the value of $I_{k}$, if $I_{k}=1$, then
the player whose turn it is may pass his turn to his opponent or use it and
observe the variable $I_{k+1}$. When the last event is reached, if the
result is success ($I_{n}=1$), the player whose turn it is wins, and loses
otherwise. Specifically, if $I_{i}=0$ for all $i$, player A loses. This is
reminiscent of the \emph{hot potato game} in which the goal is not to be holding
the hot potato at the end of the game, with the rule of being able to pass
it on (if one so wishes) to one's opponent when $I_{k}=1$.

Let us denote by $\mathbf{V}_{k}$ the probability of the player whose turn
it is winning when we are about to observe the variable $I_{k}$. In
particular, the probability of player A winning is $\mathbf{V}_{1}$; hence
the probability of player B winning is $1-\mathbf{V}_{1}$. Likewise, on
observing the last random variable, the player whose turn it is will win
with probability $p_{k}$, i.e. $\mathbf{V}_{n}=p_{n}$.

The dynamic program to find the optimal strategy is straightforward. After
observing the variable $I_{k}$, if $I_{k}=0$, which occurs with probability $%
1-p_{k}$, the player then irrevocably goes on to observe the variable $%
I_{k+1}$ without giving up his turn. If $I_{k}=1$, the optimal strategy
of the player whose turn it is will consist in passing his turn to his
opponent if $\mathbf{V}_{k+1}<\frac{1}{2}$ and in continuing with his turn if $%
\mathbf{V}_{k+1}\geq \frac{1}{2}$. We shall then have the following recurrence.

\begin{equation*}
\mathbf{V}_k=p_k \cdot \max\{\mathbf{V}_{k+1},1-\mathbf{V}_{k+1}\}+(1-p_k) \cdot \mathbf{%
V}_{k+1}; \mathbf{V}_n=p_n.
\end{equation*}

\section{Optimal strategy}

We shall see that the optimal strategy is extremely simple and that it is
also very easy to determine which of the two players has the greatest
probability of winning. Another matter altogether is the exact calculation
of this probability, which generally requires the computation of recurrence
or calculations of the equivalent cost.

\begin{prop}
\label{main}If for all $k\in \lbrack r,n]$, $p_{k}<\frac{1}{2}$, then for all $k\in
\lbrack r,n]$ the following is fulfilled:
\begin{equation*}
p_{n}=\mathbf{V}_{n}<\mathbf{V}_{k+1}<\mathbf{V}_{k}<\frac{1}{2}.
\end{equation*}
\end{prop}

\begin{proof}
It is evident that $\mathbf{V}_{n}=p_{n}<\frac{1}{2}$.  We proceed by backward induction. We assume that the proposition is true  for all $i\in \lbrack k+1,n]$ and shall prove that it also holds for $i=k$.
%\begin{equation*}
%\mathbf{V}_{k}=p_{k}\cdot \max \{\mathbf{V}_{k+1},1-\mathbf{V}_{k+1}\}+(1-p_{k})\cdot \mathbf{V}_{k+1},
%\end{equation*}%
From the induction hypothesis, $\mathbf{V}_{k+1}<\frac{1}{2}$, therefore $1-\mathbf{V%
}_{k+1}>\frac{1}{2}>\mathbf{V}_{k+1}$, and hence
\begin{equation*}
\mathbf{V}_k = p_k (1-\mathbf{V}_{k+1})+(1-p_k) \mathbf{V}_{k+1}> p_k
\mathbf{V}_{k+1} +(1-p_k) \mathbf{V}_{k+1}=\mathbf{V}_{k+1}.
\end{equation*}
On the other hand, considering that  $$\textrm{if }x,y \in [0,1/2) \textrm{  then }x(1-y)+(1-x)y<\frac{1}{2},$$
it turns out that
\begin{equation*}
\mathbf{V}_k = p_k (1-\mathbf{V}_{k+1})+(1-p_k) \mathbf{V}_{k+1}<\frac{1}{2}%
.
\end{equation*}
\end{proof}

\begin{prop}
\label{main}If $p_{r}>\frac{1}{2} $ and for all $k\in \lbrack r+1,n]$, $p_{k}<\frac{1}{2}$, then:
\begin{equation*}
 \mathbf{V}_{1}=...=\mathbf{V}_{r-1}=\mathbf{V}_{r}>\frac{1}{2}.
\end{equation*}
\end{prop}

\begin{proof} We will take into account that$$\textrm{If }x \in (1/2,1]\textrm{ and } y \in [0,1/2)\textrm{  then }x(1-y)+(1-x)y>\frac{1}{2}.$$
For   Proposition \ref{main}, $\mathbf{V}_{r+1}<\frac{1}{2}$, so

$$\mathbf{V}_{r}=p_{r} \cdot \max \{\mathbf{V}_{r+1},1-\mathbf{V}_{r+1}\}+(1-p_{r}) \cdot\mathbf{V}_{r+1}$$

$$=p_{r} \cdot (1-\mathbf{V}_{r+1})+(1-p_{r}) \cdot\mathbf{V}_{r+1}> \frac{1}{2}$$

Now, since  $\mathbf{V}_{r}>\frac{1}{2}$, it is immediately followed by the dynamic program that $$\mathbf{V}_{1}=...=\mathbf{V}_{r-1}=\mathbf{V}_{r}.$$

\end{proof}

\begin{prop}
\label{main}If $p_{r}=\frac{1}{2} $ and for all $k\in \lbrack r+1,n]$, $p_{k}<\frac{1}{2}$, then:
\begin{equation*}
\mathbf{V}_{1}=...=\mathbf{V}_{r-1}=\mathbf{V}_{r}=\frac{1}{2}.
\end{equation*}
\end{prop}

\begin{proof} For   Proposition \ref{main}, $\mathbf{V}_{r+1}<\frac{1}{2}$, so

$$\mathbf{V}_{r}=\frac{1}{2} \cdot \max \{\mathbf{V}_{r+1},1-\mathbf{V}_{r+1}\}+(1-\frac{1}{2}) \cdot\mathbf{V}_{r+1}=\frac{1}{2}$$
 
Now, since $\mathbf{V}_{r}=\frac{1}{2}$, it follows immediately from the dynamic program that $$\mathbf{V}_{1}=...=\mathbf{V}_{r-1}=\mathbf{V}_{r}.$$

\end{proof}

From the previous propositions the following result is followed without difficulty.

\begin{prop}
Let $\Omega _{r}:=\{k\in \lbrack r,n]:p_{k}\geq \frac{1}{2}\}$ and considering
\begin{equation*}
\mathfrak{u}_{r}:=%
\begin{cases}
\max \Omega , & \text{if $\Omega _{r}\neq \emptyset $ }; \\
r, & \text{if $\Omega _{r}=\emptyset .$ }%
\end{cases}%
\end{equation*}%
The optimal strategy for the player whose turn it is when observing the
variable $I_{r}$ is not to give up his turn before stage $\mathfrak{u}_{r}$
and to do so when he may starting from $\mathfrak{u}_{r}$. In addition, the following is true.

$\bullet$ If $\Omega_r=\emptyset$, then $\mathbf{V}_r <\frac{1}{2}$.

$\bullet$ If $\Omega_r\neq\emptyset $ and $p_{\mathfrak{u}_r}=\frac{1}{2}$, then $%
\mathbf{V}_r=\frac{1}{2}$.

$\bullet$ If $\Omega_r\neq\emptyset $ and $p_{\mathfrak{u}_r}>\frac{1}{2}$, then $%
\mathbf{V}_r>\frac{1}{2}$.
\end{prop}

%\begin{proof}
%$\bullet $ If $\Omega _{r}=\emptyset $, from Proposition \ref{main} we have
%that $\mathbf{V}_{k+1}<\mathbf{V}_{k}<\mathbf{V}_{r}<\frac{1}{2}$ for all $k\in \lbrack r+1,n]$. Thus, it is
%always preferable for the player to give up his turn after any stage $k\geq r
%$ than to continue with his turn. If the player continues with his turn, the
%win probability is $\mathbf{V}_{k+1}<\frac{1}{2}$, while if he gives up his turn, it is
%greater.
%
%$\bullet $ If $\Omega _{r}\neq \emptyset $, we have that $p_{\mathfrak{u}%
%_{r}}\geq \frac{1}{2}$ and $p_{k}<\frac{1}{2}$ for all $k\in \lbrack \mathfrak{u}%
%_{r}+1,n]$. Hence, similar reasoning as above may be applied.
%
%If $p_{u_{r}}=\frac{1}{2}$, then
%\begin{equation*}
%\mathbf{V}_{r}=\mathbf{V}_{\mathfrak{u}_r}=\frac{1}{2}\max (\mathbf{V}_{\mathfrak{u}_r+1},1-\mathbf{V}_{\mathfrak{u}_r+1})+\frac{1}{2}\mathbf{V}_{\mathfrak{u}_r+1}=\frac{1}{2} (1-\mathbf{V}_{\mathfrak{u}_r+1})+\frac{1}{2}\mathbf{V}_{\mathfrak{u}_r+1}=\frac{1}{2}
%\end{equation*}%
%
%
%If $p_{u_{r}}>\frac{1}{2}$, then
%\begin{equation*}
%\mathbf{V}_{r}=p_{u_{r}}\max (\mathbf{V}_{r+1},1-\mathbf{V}_{r+1})+(1-p_{u_{r}})\mathbf{V}_{r+1}
%\end{equation*}%
%\begin{equation*}
%\mathbf{V}_{r}=p_{u_{r}}(1-\mathbf{V}_{r+1})+(1-p_{u_{r}})\mathbf{V}_{r+1}<\frac{1}{2}
%\end{equation*}
%\end{proof}

Let us denote by $\mathfrak{u}$ the optimal threshold of the first player in
his first turn, $\mathfrak{u}:=\mathfrak{u}_{1}$ (the last Bernoulli event
with parameter $\geq \frac{1}{2}$). The
optimal strategy of the first player consists in continuing with his turn
until reaching the $\mathfrak{u}$-th event and thereafter giving up his turn
whenever possible. Obviously, player B will do the same in his optimal game
because, when his turn comes, he will be in the same situation as player A.
In short, we have the following result.

\begin{teor}
The optimal strategy for both players is to give up their turn when (and
only when) there are no random variables left to observe whose parameter is
greater than or equal to $\frac{1}{2}$.
\end{teor}

In fact, when played optimally by both players, the game can be seen as a
game of solitaire played by player A assuming his opponent uses the optimal
strategy. Thus, the probability of player A winning, the optimal threshold
being $\mathfrak{u}$, is the probability that the number of $1's$
starting from the $\mathfrak{u}$ resulting from the random variants is odd;
that is to say:

\begin{equation*}
\mathbf{V}_{1}=P\left( \sum_{i=\mathfrak{u}}^{n}I_{i}=odd\right)
\end{equation*}%
%Hence the following result:

%\begin{prop}
%\label{suma} The probability of  player $A$ winning, when both players use
%the optimal strategy, is
%\begin{equation*}
%\mathbf{V}_{1}=\sum_{i=\mathfrak{u}}^{\lceil \frac{n}{2}-1\rceil }{\left(
%1-p_{i}\right) }^{-1-2\,i+n}\,p_{i}^{1+2\,i}\,{\binom{n}{1+2i}}.
%\end{equation*}
%
%\begin{equation*}
%\mathbf{V}_{1}=\sum_{i=\mathfrak{u}}^{\lceil \frac{n}{2}-1\rceil }{\left(
%1-p_{i}\right) }^{-\mathfrak{u} -1-2\,i+n}\,p_{i}^{\mathfrak{u}+1+2\,i}\,{\binom{n}{1+2i}}.
%\end{equation*}
%\end{prop}

This  allows establishing a (somewhat coarse) lower
bound for the probability of player A winning. Bear in mind that the win
probability in this game is greater than the probability of winning in the
LSP.

\begin{prop}
If $\sum_{i=1}^{n}\frac{p_{i}}{1-p_{i}}\geq 1$, then
\begin{equation*}
\mathbf{V}_{1}>\frac{1}{e}.
\end{equation*}
\end{prop}

\begin{proof}
It suffices to keep in mind that the probability of winning in the LSP under
these conditions is greater than $1/e$ (see \cite{BR2}).
\end{proof}

\section{All the random variables have the same parameter}

In this section, we study the particular case that all the Bernoulli random
variables have the same parameter.

\begin{prop}
If $p_{i}=p$ for all $i=1,...,n$, then the probability of player A winning  is strictly increasing with $n$ always below
its limit as $n$ tends to infinity

\begin{equation*}
\mathbf{V}_1 =\frac{1 - {\left( 1 - 2\,p \right) }^ n}{2}<\frac{1}{2}.
\end{equation*}

\end{prop}

\begin{proof}
%We take Proposition \ref{suma} into consideration.

If $n$ is even, we have

\begin{equation*}
\mathbf{V}_1=\sum_{i=0}^{\frac{n}{2}-1}{\left( 1 -p \right) }^ {-1 -
2\,i + n}\, p ^{ 1 + 2\,i}\,  {\binom{n}{1+2i}}=\frac{1 - {\left( 1 - 2\,p
\right) }^ n}{2}.
\end{equation*}

Similarly, if n is odd, we have
\begin{equation*}
\mathbf{V}_1=\sum_{i=0}^{\frac{n-1}{2}}{\left( 1-p\right) }%
^{-1-2\,i+n}\,p^{1+2\,i}\,{\binom{n}{1+2i}}=\frac{1-{\left( 1-2\,p\right) }%
^{n}}{2}.
\end{equation*}
\end{proof}

\begin{prop}
If we have $n$ Bernoulli random variables with $p_{i}=\frac{1}{n}$, the
probability of player A winning is decreasing and is always greater that its
limit as $n$ tends to infinity, namely
\begin{equation*}
\frac{1}{2}-\frac{1}{2\,e^{2}}=0.4323323...
\end{equation*}
\end{prop}

\begin{proof}
If $n=1$ then $\mathbf{V}_1=1$. If $n=2$ then $V_2=\frac{1}{2}$. If $n>3$, the optimal
strategy for both players is to give up their turn whenever possible. Hence,
player A will win if the number of $1's$ resulting from the random
variables is odd.

If $n$ is even
\begin{equation*}
\mathbf{V}_{1}=\sum_{i=0}^{\frac{n}{2}-1}{\left( 1-\frac{1}{n}\right) }%
^{-1-2\,i+n}\,\left( \frac{1}{n}\right) ^{1+2\,i}\,{\binom{n}{1+2i}}=
\end{equation*}

\begin{equation*}
=\frac{{\left( \frac{-1 + n}{n} \right) }^n\, \left( -{\left( \frac{-2 + n} {%
-1 + n} \right) }^n + {\left( \frac{n}{-1 + n} \right) }^n \right) }{2}
\end{equation*}

Similarly, if $n$ is odd
\begin{equation*}
\mathbf{V}_{1}=\sum_{i=0}^{\frac{n-1}{2}}{\left( 1-\frac{1}{n}\right) }%
^{-1-2\,i+n}\,\left( \frac{1}{n}\right) ^{1+2\,i}\,{\binom{n}{1+2i}}=
\end{equation*}

\begin{equation*}
=\frac{{\left( \frac{-1 + n}{n} \right) }^n\, \left( -{\left( \frac{-2 + n} {%
-1 + n} \right) }^n + {\left( \frac{n}{-1 + n} \right) }^n \right) }{2}
\end{equation*}

\begin{equation*}
\lim_{n\rightarrow\infty}\mathbf{V}_1=\frac{1}{2} - \frac{1}{2\,e^2}%
=0.4323323...
\end{equation*}
\end{proof}

\begin{prop}
If we have $n$ Bernoulli random variables with $p_{i}\geq \frac{1}{n}$, the
probability of player A winning is greater than
\begin{equation*}
\frac{1}{2}-\frac{1}{2\,e^{2}}=0.4323323...
\end{equation*}
\end{prop}

\begin{proof}
If $p_{i}\geq\frac{1}{2}$ for some $i$, then $V_1\geq\frac{1}{2}$. Otherwise, think of the auxiliar game with all the parameters equal to $1/n$ in which
the probability is greater than $\frac{1}{2}-\frac{1}{2\,e^{2}}$. Now, there is no more to considering successive modifications of this game, as in Lemma \ref{aumen} (see below), with which the win probability increases, until reaching the game considered.
\end{proof}

\begin{prop}
If we have $n$ Bernoulli random variables, of which there are $m$ with $%
p_{i}\geq \frac{1}{m}$, the probability of player A winning is greater than
\begin{equation*}
\frac{1}{2}-\frac{1}{2\,e^{2}}=0.4323323...
\end{equation*}
\end{prop}

\begin{proof}
It can easily be seen that if $p_{i}<\frac{1}{2}$ for all $i$, then the probability
of player A winning is greater than the probability that he would have in
the game resulting from  excluding  some random variable.
Consequently, it suffices to observe that the value,  $\frac{1}{2%
}-\frac{1}{2\,e^{2}}$, is exceeded in the auxiliary game resulting from
 excluding  some random variables. In fact, if we have $m$ variables
with $p_{i}<1/m$ for all of these variables, considering the game in which
the other variables are  excluded, then we are able to use the
previous proposition.
\end{proof}

\begin{lem}
\label{aumen} Let us consider the game with parameters $p_{i}<\frac{1}{2}$
and denote by $\overline{\mathbf{V}}_{i}$ the probability of a player
winning when it is his turn after observing the variable $I_{i}$ in the
resulting auxiliary game when changing $p_{k}$ to $\overline{p}_{k}>p_{k}$%
. Hence,

\begin{equation*}
\text{For all }i\in [k+1,n] , \overline{\mathbf{V}}_i=\mathbf{V}_i
\end{equation*}

\begin{equation*}
\text{For all }i\in [1,k] , \overline{\mathbf{V}}_i>\mathbf{V}_i
\end{equation*}

In other words, if we increase the value of the parameter of one of the
Bernoulli random variables in a game, then the player's probability of
winning on his turn increases.
\end{lem}

\begin{proof}
Let us recall that $\mathbf{V}_{i}$ and $\overline{\mathbf{V}}_{i}$
respectively denote the player's probability of winning on his turn at stage
$i$ in the original game and in the auxiliary game.

$\bullet $ For all $i>k$, it is evident that $\mathbf{V}_{i}=\overline{%
\mathbf{V}}_{i}$ as we are in a subsequent stage to the modified variable
and the process \textquotedblleft has no memory\textquotedblright\ and
therefore does not affect.

$\bullet $ For all $i\leq k$, we will proceed by induction backwards. Let us
first see that it is true for $i=k$.

\begin{equation*}
\mathbf{V}_k=p_k (1-\mathbf{V}_{k+1})+(1-p_k) \mathbf{V}_{k+1}
\end{equation*}

\begin{equation*}
\overline{\mathbf{V}}_k=\overline{p}_k (1-\mathbf{V}_{k+1})+(1-\overline{p}%
_k) \mathbf{V}_{k+1}
\end{equation*}

\begin{equation*}
\overline{\mathbf{V}}_k-\mathbf{V}_k=(\overline{p}_k-p_k)(1-\mathbf{V}%
_{k+1})-(\overline{p}_k-p_k) \mathbf{V}_{k+1}=(\overline{p}_k-p_k)(1-2%
\mathbf{V}_{k+1})>0
\end{equation*}

We now assume that the proposal is fulfilled for $i+1$ and shall prove that
it is fulfilled for $i$

\begin{equation*}
\mathbf{V}_i=p_i (1-\mathbf{V}_{i+1})+(1-p_i) \mathbf{V}_{i+1}
\end{equation*}

\begin{equation*}
\overline{\mathbf{V}}_i=p_i (1-\overline{\mathbf{V}}_{i+1})+(1-p_i)
\overline{\mathbf{V}}_{i+1}
\end{equation*}

\begin{equation*}
\overline{\mathbf{V}}_i-\mathbf{V}_i=p_i(\mathbf{V}_{i+1}-\overline{\mathbf{V%
}}_{i+1})-(1-p_i)(\mathbf{V}_{i+1}-  \overline{\mathbf{V}}_{i+1})=(2p_i-1)(%
\mathbf{V}_{i+1}-\overline{\mathbf{V}}_{i+1})>0
\end{equation*}
\end{proof}

\subsection{A variant: If there have been no $1^{\prime }s$, the game is
repeated}

We have seen that the game is advantageous for player A if and only if some
parameter is greater than $\frac{1}{2}$. The reason that the game is disadvantageous
for player A is related to the fact that he can lose because the results of
all the random variables are $0$. In fact, if any of the variables
is worth $1$, then the probability of the player winning by giving up his
turn is greater than $\frac{1}{2}$. The following result shows that, if the rule of
repeating the game is introduced and if there have been no $1's$,
then the game is very advantageous for player A.

\begin{prop}
If $p_{i}=1/n$ for all $i$ and considering the rule that the game is
repeated in the case of $I_{i}=0$ for all $i$, the probability of player A
winning is
\begin{equation*}
\mathcal{V}(n):=\frac{{\left( -2+n\right) }^{n}-{\left( -1+n\right) }^{n}\,{%
\left( \frac{n}{-1+n}\right) }^{n}}{2\,{\left( -1+n\right) }^{n}-2\,n^{n}}.
\end{equation*}%
Besides, we have that $\mathcal{V}(n)$ is increasing and

\begin{equation*}
\frac{2}{3}= \mathcal{V}(2) < {V}(n)<\frac{1 + e}{2\,e}= 0.683939...
\end{equation*}
\end{prop}

\begin{proof}
Obviously, the optimal strategy with this rule is the same as in the game in
its original version. The difference lies only in the probability of
winning, which is conditioned by $\sum_{i=1}^{n}I_{i}>0$. Thus, bearing in
mind that
\begin{equation*}
P\left( \sum_{i=1}^{n}I_{i}>0\right) =1-P\left(
I_{1}=I_{2}=...=I_{n}=0\right) =1-\left( 1-\frac{1}{n}\right) ^{n}
\end{equation*}

we have that
\begin{equation*}
\mathcal{V}(n)=\frac{{\left( \frac{-1 + n}{n} \right) }^n\, \left( -{\left(
\frac{-2 + n} {-1 + n} \right) }^n + {\left( \frac{n}{-1 + n} \right) }^n
\right) }{2} \cdot \frac{1}{1- (1-\frac{1}{n})^n}
\end{equation*}

Moreover, $\mathcal{V}(n)$ is increasing and its limit is $\frac{1+e}{2\,e}$.
\end{proof}

\section{Random parameters for the Bernoulli variables}

We finally determine the probability of player A winning (mean probability)
when the parameters are the results of uniform random variables, $\mathbf{U}%
[0,1]$. That is to say, before holding the competition, the parameters of
the Bernoulli variables are drawn via random trials of a uniform random
variable $\mathbf{U}[0,1]$ and these parameters are revealed to the players.

\begin{lem}
If the last $k$ variables $\{I_{i}\}_{i=n-k+1}^{n}$ have parameters $p_{i}$,
which are the results of the uniform random variables $\mathbf{U}[0,\frac{1}{2}]$,
then the win probability of the player whose turn it is at stage $n-k+1$ is
\begin{equation*}
\mathbf{V}_{n-k+1}=2^{-1-k}\,\left( -1+2^{k}\right)
\end{equation*}%
In particular, if $k=n$

\begin{equation*}
\mathbf{V}_{ 1}=2^{-1 - n}\,\left( -1 + 2^n \right)
\end{equation*}
\end{lem}

\begin{proof}
We denote by $X_{i}$ the win probability of player whose turn it is at stage
$n-i+1$. Bearing in mind that $t=p_{n-k+1}$ is the result of a uniform
random variable, $\mathbf{U}[0,\frac{1}{2}]$
\begin{equation*}
X_{i}=\mathbb{E}(tX_{i-1}+(1-t)(1-X_{i-1}))
\end{equation*}

\begin{equation*}
X_{i}=2\int_{0}^{\frac{1}{2}}\left( tX_{i-1}+(1-t)(1-X_{i-1})\right) dt=%
\frac{1}{4}(1+2X_{i-1})
\end{equation*}%
and solving with $X_{0}=0$ gives
\begin{equation*}
X_{k}=2^{-1-k}\,\left( -1+2^{k}\right)
\end{equation*}
\end{proof}

\begin{lem}
If we have $k-1$ end variables with parameters $p_{i}$ that are the result
of uniform random variables, $\mathbf{U}[0,\frac{1}{2}]$, and $p_{k}$ is the result
of a uniform random variable, $\mathbf{U}[\frac{1}{2},1]$, then the probability of
player A winning is
\begin{equation*}
J_{k}:=\frac{1+2^{-k}}{2}
\end{equation*}
\end{lem}

\begin{proof}
We denote by $J_{k}$ the win probability of the player whose turn it is
after $n-k+1$. Reasoning similar to above

\begin{equation*}
J_k= \mathbb{E}(t X_{k-1}+ (1-t) (1-X_{k-1})
\end{equation*}
\begin{equation*}
J_k=2 \int_{\frac{1}{2}}^1 \left(t X_{k-1}+ (1-t) (1-X_{k-1})\right) dt
\end{equation*}
\begin{equation*}
=2 \int_{\frac{1}{2}}^1 2^{-1 - k}\,\left( -2 + 2^k + 4\,x \right) dt=\frac{1}{2}%
(1+2^{-k})
\end{equation*}
\end{proof}

\begin{lem}
If all the parameters $p_{i}$ are the result of uniform random variables, $%
\mathbf{U}[0,\frac{1}{2}]$, then player A's probability of winning is
\begin{equation*}
J_{k}:=\frac{1+2^{-k}}{2}
\end{equation*}
\end{lem}

\begin{prop}
If we have a game with $n$ variables whose parameters are the result of the $%
n$ uniform random variables, $\mathbf{U}[0,1]$, player A's probability of
winning is
\begin{equation*}
\frac{2\,\left( 1-4^{-n}\right) }{3}.
\end{equation*}
\end{prop}

\begin{proof}
The probability that the last parameter greater than or equal to $\frac{1}{2}$ will
be the $k$-th is $2^{-n+k-1}$ and the probability that all the parameters
are less than $\frac{1}{2}$ will be $2^{-n}$.

\begin{equation*}
\mathbb{E}(n)=2^{-1 - n}\,\left( -1 + 2^n \right)\cdot2^{-n}+\sum_{k=1}^n
J_k 2^{-k}=\frac{2\,\left( 1 - 4^{-n} \right) } {3}
\end{equation*}
\end{proof}

\section{Conclusions and future challenges}

The proposed adversarial version of the Last-Success-Problem has a very
simple optimal game strategy that does not require any calculation. It only
requires identifying the last variable whose parameter is greater than $\frac{1}{2}$
and, as from that point on, always giving up one's turn to one's opponent.
It seems interesting to pose the problem with non-independent random
variables. The Last-Success-Problem with dependent Bernoulli random
variables was addressed by Tamaki in \cite{tama}, who considered that $%
I_{1},...,I_{n}$ constitute a Markov chain with transition probabilities

\begin{equation*}
\alpha_j= P(I_{j+1}=1 | I_j=0)
\end{equation*}
\begin{equation*}
\beta_j= P(I_{j+1}=0 | I_j=1)
\end{equation*}

and established an optimal  stopping rule  with a \emph{Markov
version of the odds-theorem}.

We predict that the adversarial version with dependent variables will also
be simple and the optimal strategy will most likely consist in adopting, at
each $k$-th stage, the optimal strategy while assuming that the remaining
variables are independent $\widehat{I}_{k+1},...,\widehat{I}_{n}$ with
parameters (computable by recurrence)
\begin{equation*}
\widehat{p}_{i}=P(I_{i}=1|I_{k}=1)
\end{equation*}

In short, we conjecture the following.

\begin{con}
Let $I_{1},I_{2},...,I_{n}$ be $n$ dependent Bernoulli random variables. Let $p_{i,k}:=P(I_{i}=1|I_{k}=1)$. Then, the optimal
strategy for the player whose turn it is after observing the variable $%
I_{k}=1$ is to give up his turn to his opponent if and only if $p_{i,k}<\frac{1}{2}$
for all $i>k$.
\end{con}

It may also be interesting to pose the game with more than 2 players, in
which case different types of payment could be considered. For any version,
it is normal to consider the loser to be the player whose turn it is after
the last Bernoulli trial, but several types of payment may be considered for
the other  players. If we consider that the players who do not lose each
receive the same payment, we have the simplest version. In this respect, we
conclude by posing the challenge to determine the limit with $m$
players, when $n$ tends to infinity, of the loss probability of each player,
considering $n$ independent Bernoulli random variables with parameters $1/n$. In fact, for 3 players, it is no longer a trivial problem, as only the
limit for the probability of the first player losing is exactly
calculable in a relatively straightforward way. Using the \textsf{Mathematica }symbolic calculation package, we obtained  the following limit for the probability of the first player losing:
\begin{equation*}
loss_{1}=\frac{1}{3}+\frac{2\,\cos (\frac{\sqrt{3}}{2})}{3e^{\frac{3}{2}}}%
=0.42970463...
\end{equation*}%
However, it is no longer viable to find the exact limit of the probability
of the other two players losing via this path. Computing for large values of
$n$ allows an approximation, but only that. Specifically, we have that:

\begin{equation*}
loss_{2}\approx 0.383\text{ and }loss_{3}\approx 0.187
\end{equation*}%
In all the above calculations, we have assumed that the optimal strategy for
both players is to give up their turn whenever possible. Of course, this
will undoubtedly be true in this case. In general, however, there will be an
optimal strategy that does not always consist in passing one's turn to the
following player.

\end{document}